\documentclass{article}
\usepackage{amsmath}
\usepackage{amsthm}
\usepackage{mathrsfs}
\usepackage{latexsym}
\usepackage{amssymb}
\usepackage{amscd}
\usepackage[dvips]{graphics}
\usepackage[all,cmtip]{xy}

\DeclareSymbolFont{AMSb}{U}{msb}{m}{n}
\DeclareMathSymbol{\N}{\mathbin}{AMSb}{"4E}
\DeclareMathSymbol{\Z}{\mathbin}{AMSb}{"5A}
\DeclareMathSymbol{\R}{\mathbin}{AMSb}{"52}
\DeclareMathSymbol{\Q}{\mathbin}{AMSb}{"51}
\DeclareMathSymbol{\I}{\mathbin}{AMSb}{"49}
\DeclareMathSymbol{\C}{\mathbin}{AMSb}{"43}

\DeclareFontFamily{U}{mathx}{\hyphenchar\font45}
\DeclareFontShape{U}{mathx}{m}{n}{
      <5> <6> <7> <8> <9> <10>
      <10.95> <12> <14.4> <17.28> <20.74> <24.88>
      mathx10
      }{}
\DeclareSymbolFont{mathx}{U}{mathx}{m}{n}
\DeclareFontSubstitution{U}{mathx}{m}{n}
\DeclareMathAccent{\widecheck}{0}{mathx}{"71}
\DeclareMathAccent{\wideparen}{0}{mathx}{"75}

\DeclareMathOperator*{\ctens}{\widehat{\otimes}}
\DeclareMathOperator*{\cotens}{\square}

\newcommand{\iso}{\cong}
\newcommand{\invlim}{\underleftarrow{\textnormal{lim}}\,}
\newcommand{\dirlim}{\underrightarrow{\textnormal{lim}}\,}

\newcommand{\dash}{\textnormal{-}}
\newcommand{\tn}[1]{\textnormal{#1}}
\newcommand{\cat}[1]{\tn{\textbf{#1}}}

\numberwithin{equation}{section}

\title{Comodules, contramodules and Pontryagin duality}

\author{John MacQuarrie and Ricardo Souza}

\begin{document}

\newtheorem{defn}[equation]{Def{i}nition}
\newtheorem{prop}[equation]{Proposition}
\newtheorem{lemma}[equation]{Lemma}
\newtheorem{theorem}[equation]{Theorem}
\newtheorem{corol}[equation]{Corollary}
\newtheorem{question}[equation]{Questions}

\maketitle

\section{Introduction}

Let $k$ be a field, treated as a discrete topological ring.  Whenever $V$ is a topological vector space, denote by $V^*$ the vector space $\tn{Hom}_k(V,k)$ of continuous linear maps from $V$ to $k$, given the compact-open topology.  Tensor products are taken over $k$ if not otherwise indicated.

Let $C$ be a $k$-coalgebra \cite[\S 1.1]{DNR}.  In \cite{Simson}, Simson showed that the topological vector space $C^*$ inherits the structure of a pseudocompact algebra and that the contravariant functor $(-)^* = \tn{Hom}_k(-,k)$ yields a duality between the category of coalgebras and coalgebra homomorphisms, and the category of pseudocompact algebras and continuous algebra homomorphisms \cite[Theorem 3.6]{Simson}.  Consider the following categories:
\begin{itemize}
\item $C\dash\cat{Comod}$ the category of (discrete) left $C$-comodules and comodule homomorphisms \cite[\S 2.1]{DNR}.  By the fundamental theorem of comodules \cite[Theorem 2.1.7]{DNR}, every object of this category is a direct limit of finite dimensional comodules.
\item $C^*\dash\cat{PMod}$ the category of pseudocompact left $C^*$-modules and continuous module homomorphisms.  Such modules are inverse limits of finite dimensional discrete left $C^*$-modules \cite[\S 1]{Brumer}.
\item $\cat{DMod}\dash C^*$ the category of discrete right $C^*$-modules and module homomorphisms.  Such modules are direct limits of finite dimensional right $C^*$-modules \cite[\S 1]{Brumer}.
\end{itemize}

In the following diagram each arrow is a duality.  The image of the object $X$ under either arrow (or its inverse) is $X^*$ as a vector space, imbued with an appropriate additional structure:
$$\xymatrix{
C^*\dash\cat{PMod} \ar[rr] \ar[d] & & \cat{DMod}\dash C^* \\
C\dash\cat{Comod}
}$$
The horizontal arrow is an instance of the famous Pontryagin duality \cite[Proposition 2.3]{Brumer}.  The vertical arrow is due to Simson \cite[Theorem 4.3]{Simson}.

\medskip

The notion of a contramodule for the coalgebra $C$ was first introduced by Eilenberg and Moore \cite[\S III.5]{EilenbergMoore}.  Although they have received far less attention than comodules, they have been further studied by Positselski (\cite{Pos11}, for instance) and others (\cite{BBW09,UB69,Wis10}, for instance).  In this note we define the category $\cat{PContramod}\dash C$ of pseudocompact right $C$-contramodules and show that we can complete the diagram above as follows 
\begin{equation}\label{dualities}
\xymatrix{
C^*\dash\cat{PMod} \ar[rr] \ar[d] & & \cat{DMod}\dash C^*\ar[d] \\
C\dash\cat{Comod} \ar[rr] && \cat{PContramod}\dash C
}
\end{equation}
with each arrow or its inverse a duality given on objects by $X\mapsto X^*$ (imbued with appropriate additional structure).  One might think of the lower arrow as a Pontryagin duality between comodules and pseudocompact contramodules.  As an easy corollary of the main theorem, we will show that the objects of $\cat{PContramod}\dash C$ are inverse limits of finite dimensional right $C$-contramodules.  We also extend a very useful natural isomorphism of Takeuchi.

\section{Preliminaries}

Let $(C,\Delta,\varepsilon)$ be a coalgebra \cite[\S 1.1]{DNR}.  We define the category $C\dash\cat{Comod}$ explicitly in order to use it later.  A left $C$-comodule is a pair $(X,\rho)$ where $X$ is a discrete $k$-vector space and $\rho:X\to C\otimes X$ is a linear map such that the following two diagrams commute:
$$\xymatrix{
X \ar[r]^{\rho} \ar[d]_{\rho} & C\otimes X\ar[d]^{\tn{id}\otimes \rho}
& X\ar[rr]^{\sim}\ar[rd]_{\rho} && k\otimes X \\
C\otimes X\ar[r]_-{\Delta\otimes\tn{id}} & C\otimes C\otimes X 
&& C\otimes X \ar[ru]_{\varepsilon\otimes\tn{id}} & .
}$$
For convenience, we refer to the above diagrams as the ``comodule square'' and the ``comodule
triangle'', respectively.  A homomorphism of left comodules $(X,\rho)\to (Y,\gamma)$ is a linear map $\alpha: X\to Y$ such that the following ``comodule homomorphism square'' commutes:
$$\xymatrix{
X\ar[r]^{\alpha}\ar[d]_{\rho} & Y\ar[d]^{\gamma} \\
C\otimes X\ar[r]_{\tn{id}\otimes \alpha} & C\otimes Y
\,.}$$

Let $A,B$ be discrete vector spaces and let $Z$ be a pseudocompact vector space.  The map 
$$\psi_{A,Z}^B : \tn{Hom}_k(A, \tn{Hom}_k(B,Z)) \to \tn{Hom}_k(A\otimes B, Z)$$
given by $\psi_{A,Z}^B(\gamma)(a\otimes b) := \gamma(a)(b)$ is an isomorphism.  To see this, one writes $A$ and $B$ as direct limits of finite dimensional vector spaces $A_i, B_j$ respectively and $Z$ as an inverse limit of finite dimensional vector spaces $Z_k$.  Using repeatedly the isomorphism of \cite[Lemma 5.1.4(b)]{RZ} (which holds for pseudocompact vector spaces) and applying the tensor-hom adjunction, one obtains isomorphisms
\begin{align*}
\tn{Hom}_k(A, \tn{Hom}_k(B,Z)) & \iso \invlim_{i,j,k} \tn{Hom}_k(A_i, \tn{Hom}_k(B_j,Z_k)) \\
& \iso  \invlim_{i,j,k} \tn{Hom}_k(A_i\otimes B_j, Z_k) \\
& \iso \tn{Hom}_k(A\otimes B, Z).
\end{align*}
The isomorphism $\psi_{A,Z}^B$ is natural in all three variables.  Naturality in $A$ and $Z$ comes via the tensor-hom adjunction applied at the level of the finite dimensional vector spaces, while $B$ is the ``parameter'' (for more on this, see \cite[\S IV.7]{MacLane}).

\begin{defn}
A \emph{pseudocompact right $C$-contramodule} is a pseudocompact vector space $Z$ together with a continuous linear map $\theta: \tn{Hom}_k(C,Z)\to Z$ such that the following diagrams commute:
$$\xymatrix{
\tn{Hom}_k(C,\tn{Hom}_k(C,Z)) \ar[rr]^-{\tn{Hom}_k(C,\theta)}\ar[d]_{\psi_{C,Z}^C} &&  \tn{Hom}_k(C,Z)\ar[dd]^{\theta} \\
 \tn{Hom}_k(C\otimes C,Z)\ar[d]_{\tn{Hom}_k(\Delta,Z)} && \\
  \tn{Hom}_k(C,Z)\ar[rr]_-{\theta} && Z
}$$
$$\xymatrix{
Z && \tn{Hom}_k(k,Z)\ar[ll]_{\sim}\ar[ld]^{\tn{Hom}_k(\varepsilon,Z)} \\
& \tn{Hom}_k(C,Z)\ar[ul]^{\theta}
}$$
We call these diagrams the contramodule square and the contramodule triangle, respectively.   A homomorphism of pseudocompact right contramodules $(Z,\theta)\to (T,\eta)$ is a continuous linear map $\alpha: Z\to T$ such that the following ``contramodule homomorphism square'' commutes:
$$\xymatrix{
\tn{Hom}_k(C,Z)\ar[rr]^{\tn{Hom}_k(C,\alpha)}\ar[d]_{\theta} && \tn{Hom}_k(C,T)\ar[d]^{\eta} \\
Z\ar[rr]_{\alpha} && T
}$$
We denote the category of pseudocompact right contramodules and contramodule homomorphisms by $\cat{PContramod}\dash C$.
\end{defn}

\section{The lower arrow}

We define the lower horizontal arrow in (\ref{dualities}) and prove it is a duality.  The right vertical duality is then obtained by composing the others.

From now on we suppress notation, writing $(A,B)$ instead of $\tn{Hom}_k(A,B)$ and $\gamma\circ-$ (resp.\ $-\circ\gamma$) instead of $\tn{Hom}_k(A,\gamma)$ (resp.\ $\tn{Hom}_k(\gamma,B)$). Let $X$ be a discrete vector space.  We have natural isomorphisms
$$(X,C\otimes X) \xrightarrow{(-)^*} ((C\otimes X)^*,X^*) \xrightarrow{-\circ\psi_{C,k}^X} ((C,X^*), X^*).$$
% \begin{align*}
% \tn{Hom}_k(X,C\otimes X) & \iso \tn{Hom}_k((C\otimes X)^*, X^*) \\
% & = \tn{Hom}_k(\tn{Hom}_k(C\otimes X,k), X^*) \\
% & \iso \tn{Hom}_k(\tn{Hom}_k(C, X^*), X^*).
% \end{align*}
Given $\rho\in (X,C\otimes X)$, denote the corresponding element $\rho^*\psi_{C,k}^X$ of $((C, X^*), X^*)$ by $\overline{\rho}$.

\begin{theorem}
The assignment defined on objects by $(X,\rho) \mapsto (X^*, \overline{\rho})$
and on morphisms by $\alpha\mapsto \alpha^*$ yields a duality of categories
$$C\dash\cat{Comod} \to \cat{PContramod}\dash C.$$
% Let $X$ be a discrete vector space.  Given $\rho:C\to C\otimes X$, The pair $(X,\rho)$ is a left $C$-comodule if, and only if $(X^*, \overline{\rho})$ is a right $C$-contramodule.
\end{theorem}

\begin{proof}
By Pontryagin duality (a special case of \cite[Proposition 2.3]{Brumer}) the operation yields a duality from discrete $k$-vector spaces to pseudocompact vector spaces.  It thus suffices to check that $(X,\rho)$ is a comodule if, and only if, $(X^*, \overline{\rho})$ is a contramodule and that $\alpha : (X,\rho)\to (Y, \gamma)$ is a comodule homomorphism if, and only if, $\alpha^*: (Y^*, \overline{\gamma})\to (X^*,\overline{\rho})$ is a contramodule homomorphism. 

Let $\rho: X\to C\otimes X$ be a linear map and consider the following diagram:
$$\xymatrix{
(C,(C,X^*)) \ar[rr]^{\psi_{C,X^*}^C}\ar[d]_{\psi_{C,k}^X\circ-}\ar@{}[rrd]|-{(1)}  &&  (C\otimes C,X^*)\ar[rr]^{-\circ\Delta}\ar[d]_{\psi_{C\otimes C, k}^X}\ar@{}[rrd]|-{(2)}   &&    (C,X^*)\ar[d]^{\psi_{C,k}^X}        \\
(C,(C\otimes X)^*)\ar[rr]_{\psi_{C,k}^{C\otimes X}}\ar[d]_{\rho^*\circ-}\ar@{}[rrd]|-{(3)}  &&  (C\otimes C\otimes X)^*\ar[rr]_{(\Delta\otimes\tn{id})^*}\ar[d]_{(\tn{id}\otimes\rho)^*}  &&    (C\otimes X)^*\ar[d]^{\rho^*}\ar@{}[lld]|-{(4)} \\
(C,X^*)\ar[rr]_{\psi_{C,k}^X}       &&   (C\otimes X)^* \ar[rr]_{\rho^*}     && X^*
}$$
We claim that squares $(1), (2)$ and $(3)$ commute:
\begin{itemize}
\item[$(1)$] is most easily seen to commute by a simple calculation with elements.
\item[$(2)$] commutes, being the natural transformation $\psi_{-,k}^X$ applied to the map $\Delta:C\to C\otimes C$.
\item[$(3)$] commutes, being the natural transformation $\psi_{C,k}^{-}$ applied to the map $\rho: X\to C\otimes X$.
%commutes due to the naturality of $\psi$ in the second coordinate (the ``parameter") applied to the map $\rho: X\to C\otimes X$ (for more on this, see [Mac Lane Section IV.7]).
\end{itemize}
The outer square commutes if, and only if, $(X^*, \overline{\rho}) = (X^*,\rho^*\psi_{C,k}^X)$ satisfies the contramodule square and Square $(4)$ commutes if, and only if $(X,\rho)$ satisfies the comodule square (by duality).  But the morphisms in Square $(1)$ are isomorphisms and hence the outer square commutes if, and only if, Square $(4)$ commutes.

\medskip

We check next the correspondence between the comodule and contramodule triangles.  Consider the following diagram, wherein the horizontal maps are natural isomorphisms.
$$\xymatrix{
 X^*  && (k\otimes X)^*\ar[ll]\ar[dl]^{(\varepsilon\otimes \tn{id})^*} \\
    &  (C\otimes X)^*\ar[ul]^{\rho^*}  &  \\
   && \\
 X^*\ar@{=}[uuu]   && (k,X^*)\ar[uuu]_{\psi_{k,k}^X}\ar[ll]\ar[dl]^{-\circ\varepsilon} \\
 & (C,X^*)\ar[uuu]_{\psi_{C,k}^X}\ar[ul]^{\overline{\rho}} &
}$$
The square at the back clearly commutes, the left front square commutes by the definition of $\overline{\rho}$ and the right front square commutes by the naturality of $\psi_{-,k}^X$ applied to $\varepsilon$.  The vertical maps are isomorphisms, so the upper triangle commutes if, and only if, the lower triangle commutes.  But the lower is the contramodule triangle and the upper is the (dual of the) comodule triangle.

\medskip

Let $\alpha: X\to Y$ be a linear transformation between the comodules $(X,\rho)$ and $(Y, \gamma)$.  One checks that $\alpha$ is a homomorphism of comodules if, and only if, $\alpha^*$ is a homomorphism of contramodules exactly as above: dualize the coalgebra homomorphism square 
%$$\xymatrix{
%X\ar[r]^{\alpha}\ar[d]_{\rho} & Y\ar[d]^{\gamma} \\
%C\otimes X\ar[r]_{\tn{id}\otimes \gamma} & C\otimes Y
%}$$
and join it to the corresponding contramodule homomorphism square using equalities, $\psi_{C,k}^X$ or $\psi_{C,k}^Y$.  Observe that the joining squares commute and hence that the comodule homomorphism square commutes if, and only if, the contramodule homomorphism square commutes.
\end{proof}

\section{Corollaries}

\begin{corol}
The category $\cat{PContramod}\dash C$ is exactly the category of inverse limits of finite dimensional right $C$-contramodules.
\end{corol}

\begin{proof}
This is dual to the fundamental theorem of comodules \cite[Theorem 2.1.7]{DNR}.
\end{proof}

% \begin{corol}
% There is a duality $\cat{DMod}\dash C^* \to\cat{PContramod}\dash C$ given on objects by $X\mapsto X^*$ endowed with an additional structure [should calculate what] and such that the diagram
% $$\xymatrix{
% C^*\dash\cat{PMod} \ar[rr] \ar[d] & & \cat{DMod}\dash C^*\ar[d] \\
% C\dash\cat{Comod} \ar[rr] && \cat{PContramod}\dash C
% }$$
% commutes.
% \end{corol}

% \noindent \underline{Remark}:

% \begin{corol}
% The category $\cat{PContramod}\dash C$ is equivalent to the category of inverse limits of finite dimensional left $C$-comodules.
% \end{corol}

% \begin{proof}
% The category of finite dimensional vector spaces is self-dual, and hence by the theorem, the category of finite dimensional left $C$-comodules is dual to the category of finite dimensional right $C$-contramodules.
% \end{proof}

A discrete right module for the pseudocompact algebra 
$$(A,m : A\ctens A \to A,u : k\to A)$$ 
(where $\ctens = \ctens_k$ is the completed tensor product \cite[\S 2]{Brumer}) is usually defined to be a discrete vector space $X$ together with a continuous map $X\times A\to X$ (cf.\ \cite[\S 5.1]{RZ}) satisfying the obvious axioms.  One familiar with coalgebras might complain that the object $X\times A$ lives neither in the discrete nor the pseudocompact category, making it difficult to interpret the multiplication as a map from a tensor product.  When dealing with dualities and coalgebras, it is perhaps easier to define a discrete right $A$-module as a discrete vector space $X$ together with a linear map $\theta:X\to \tn{Hom}_k(A,X)$ such that the diagrams
$$\xymatrix{
X \ar[r]^{\theta}\ar[d]_{\theta} & (A,X)\ar[d]^{\theta\circ -} & (k,X) && X\ar[ll]_{\sim}\ar[ld]^{\theta} \\
(A,X)\ar[r] & (A,(A,X)) &&  (A,X)\ar[ul]^{-\circ u}  &
}$$
commute, wherein the unmarked arrow is the composition 
$$(A,X) \xrightarrow{-\circ m} (A\ctens A, X) \to (A,(A,X))$$
coming from the isomorphism \cite[Lemma 2.4]{Brumer}.  This definition is equivalent to the standard one.  Interpreted this way, a homomorphism from a pseudocompact module to a discrete module
$$\gamma : (Y,\rho:Y\ctens A\to Y) \to (X, \theta : X\to (A,X))$$
is a continuous map $\gamma : Y\to X$ such that
$$\theta\gamma(y)(a) = \gamma\rho(y\ctens a)$$
for each $y\in Y, a\in A$.  %This follows by applying the normal definition of $A$-module homomorphism and using the obvious $\theta$ coming from the map defining a normal $A$-module structure on $X$.
This equality can be written without elements: $\gamma$ (as above) is a homomorphism if $\psi(\gamma\rho) = \theta\gamma$, where $\psi$ is the natural isomorphism
$$(Y\ctens A, X)\to (Y, (A,X))$$
%Observe that dualizing this isomorphism we obtain:
%$$(X^* , (A\ctens Y)^*) \iso ((A,X)^*, Y^*).$$
%I believe we have natural isomorphisms $(A\ctens Y)^* = A^*\otimes Y^*$ and $(A,X)^* = (A^*,X^*)$ (see Gchat with Ricardo).  The dual of a PC module is a comodule and the dual of a discrete module is a contramodule, so this isomorphism is 
%$$(X^* , A^*\otimes Y^*) \iso ((A^*,X^*), Y^*)$$
%and I think says that given a PC contramodule $N$ and a comodule $D$, we have a natural isomorphism
%$$\psi:((C,N),D)\to (N, C\otimes D).$$
%[should be very careful with sides].  Thus we can define:
again coming from \cite[Lemma 2.4]{Brumer}.  Observe that, writing $X = Z^*$ for some pseudocompact vector space $Z$ and noting that
$$(A\ctens Z)^* = (\invlim_{i,j} A_i\otimes Z_j)^* = \dirlim_{i,j} A_i^* \otimes Z_j^* = A^*\otimes Z^*,$$
we have 
$$(A,X)^* = (A,Z^*)^* = (A\ctens Z)^{**} = (A^*\otimes Z^*)^* = (A^*,Z^{**}) = (A^*,X^*).$$
We can thus dualize the above notion of homomorphism:
\begin{defn}
Let $C$ be a coalgebra, let $(N,\theta:(C,N)\to N)$ be a right pseudocompact $C$-contramodule and $(M,\mu:M\to M\otimes C)$ be a right C-comodule.  A \emph{homomorphism} $\gamma:(N,\theta) \to (M,\mu)$ is a continuous linear map $N\to M$ such that
$$\overline{\psi}(\gamma\theta) = \mu\gamma,$$
where $\overline{\psi}:((C,N),M)\to (N, M\otimes C)$ is obtained from $\psi$ above by duality.  
\end{defn}

% \medskip

% To prove the next result, we need to define what is a homomorphism from a pseudocompact contramodule to a comodule, show that $\alpha\in (PC,D)$ is a module homomorphism iff $\alpha^*\in (D^*,PC^*)$ is a homomorphism from a contramodule to a comodule, and then just apply Brumer's natural isomorphism.  Various things need to be checked along the way, including that $(S\ctens_A T)^* = S^*\cotens_{A^*} T^*$ (or something along those lines).

% \medskip

Let $C,D$ be coalgebras, $L$ a right $C$-comodule, $M$ a $C\dash D$-bicomodule, and $N$ a right pseudocompact $D$-contramodule.  The vector space $\tn{Hom}_D(N,M)$ inherits a natural left $C$-comodule structure (by dualizing twice the discrete right $C^*$-module $\tn{Hom}_{D^*}(M^*,N^*)$, for instance) and so $\tn{Hom}_{D}(N,M)^*$ is a right $C$-contramodule.  We define this way the functor
$$h(M,-) := \tn{Hom}_D(-,M)^* : \cat{Contramod}\dash D \to \cat{Contramod}\dash C,$$
a variant of the ``cohom'' functor of Takeuchi \cite[\S 1]{Takeuchi}.  In the category of comodules, the cotensor product functor $M\cotens_D - : D\dash\tn{Comod}\to C\dash\tn{Comod}$ has a left adjoint if, and only if, $M$ is quasi-finite as a left $C$-comodule, and in this case the corresponding left adjoint is the cohom functor of Takeuchi \cite[Proposition 1.10]{Takeuchi}.  By allowing pseudocompact contramodules, we obtain a natural isomorphism without quasi-finiteness conditions on $M$:

\begin{corol}
Let $C,D$ be coalgebras, $L$ a right $C$-comodule, $M$ a $C\dash D$-bicomodule, and $N$ a right pseudocompact $D$-contramodule.  There is a natural isomorphism of vector spaces
$$\tn{Hom}_{D}(N, L\cotens_C M) \iso \tn{Hom}_{C}(h(M,N), L).$$
\end{corol}

\begin{proof}
In light of the discussion above and the fact that
$$L\cotens_C M \iso (L^*\ctens_{C^*}M^*)^*,$$ 
the result is dual to \cite[Lemma 2.4]{Brumer}.
\end{proof}

% Let $C,D$ be coalgebras and $X$ a $C\dash D$-bicomodule.  It is well understood that the cotensor product functor $X\square_D - : D\dash\tn{Comod}\to C\dash\tn{Comod}$ [Tak?] has a left adjoint if, and only if, $X$ is quasi-finite as a left $C$-comodule.  This is because the candidate left adjoint is the ``cohom'' functor, defined on the $C$-comodule $Y$ as $h_C(X,Y):= \tn{Hom}_C(Y,X)^*$, but $\tn{Hom}_C(Y,X)^*$ has the structure of a comodule for every comodule $Y$ if, and only if, $X$ is quasi-finite as a left $C$-comodule.  On the other hand, given any $C\dash D$-bicomodule $X$ and any left $C$-comodule $Y$, $\tn{Hom}_C(Y,X)^*$ inherits naturally the structure of a left $D$-contramodule.


\begin{thebibliography}{}

\bibitem{BBW09}
Gabriella B\"ohm, Tomasz Brzezi\'nski, and Robert Wisbauer.
\newblock Monads and comonads on module categories,
\newblock {\em Journal of Algebra}, 322(5): 1719--1747, 2009.

\bibitem{Brumer} 
A.\ Brumer.
\newblock {Pseudocompact algebras, profinite groups and class formations}, 
\newblock {\em Journal of Algebra}, 4: 442--470, 1966.

\bibitem{DNR}
S.\ Dascalescu, C.\ Nastasescu and S.\ Raianu.
\newblock {\em Hopf algebras. An introduction}, volume~235 of {\em Monographs and Textbooks in Pure and Applied Mathematics}.
\newblock Marcel Dekker, Inc., New York, 2001.

\bibitem{EilenbergMoore}
S.\ Eilenberg and J.C.\ Moore.
\newblock {\em Foundations of relative homological algebra}, {\em Memoirs of the American Mathematical Society}.
\newblock American Mathematical Society, Rhode Island, 1966.

\bibitem{MacLane}
Saunders Mac~Lane.
\newblock {\em Categories for the working mathematician}, volume~5 of {\em Graduate Texts in Mathematics}.
\newblock Springer-Verlag, New York, second edition, 1998.

\bibitem{Pos11}
Leonid Positselski.
\newblock Two kinds of derived categories, {K}oszul duality, and
  comodule-contramodule correspondence.
\newblock {\em Memoirs of the AMS}, 212(996):vi+133, 2011.

\bibitem{RZ} 
L.\ Ribes and P.A.\ Zalesskii.
\newblock {\em Profinite Groups}.
\newblock Springer, Heidelberg, 2010.


% \bibitem{Montgomery}
% Susan Montgomery.
% \newblock {\em Hopf Algebras and their Action on Rings}, volume~82 of {\em Conference Board of the Mathematical Sciences, Regional Conference Series in Mathematics}.
% \newblock American Mathematical Society, Providence, Rhode Island 1993.

\bibitem{Simson}
Daniel Simson.
\newblock Coalgebras, Comodules, pseudocompact algebras and tame comodule type.
\newblock {\em Colloquium Mathematicum}, 90(1): 101--150, 2001.

% \bibitem{Sweedler}
% Moss E. Sweedler.
% \newblock {\em Hopf algebras}, {\em Mathematics Lecture Note Series}.
% \newblock W. A. Benjamin, Inc., New York, 1969.


\bibitem{UB69}
Hiroshi Uehara and Frank Brenneman.
\newblock On a cotriple homology in a fibred category.
\newblock {\em Publications of the Research Institute for Mathematical Sciences}, 4:13--38, 1968/1969.

\bibitem{Takeuchi}
Mitsuhiro Takeuchi.
\newblock Morita theorems for categories of comodules.
\newblock {\em Journal of the Faculty of Science, the University of Tokyo}, 24: 629--644, 1977.

\bibitem{Wis10}
Robert Wisbauer.
\newblock Comodules and contramodules.
\newblock {\em Glasgow Mathematical Journal}, 52(A):151--162, 2010.

\end{thebibliography}
\end{document}